\documentclass[11pt]{amsart}
\usepackage{amssymb,latexsym,cite}
\usepackage{pstcol,pstricks,color}
\usepackage{amsfonts, amsmath}
\usepackage{graphicx}

\advance\textheight by \topskip   \numberwithin{equation}{section}

\newtheorem{theorem}{Theorem}[section]

\newtheorem{corollary}[theorem]{Corollary}

\newtheorem{lemma}[theorem]{Lemma}

\newtheorem{example}[theorem]{Example}

\def\mm{{\mathbf m} }

\DeclareMathOperator{\E}{\mathbf{E}}

\DeclareMathOperator{\bP}{\mathbf P}

\begin{document}
\title[Longest increasing subsequences]{Permutations avoiding 312 and another pattern, Chebyshev polynomials and longest increasing subsequences}
\author{Toufik Mansour}
\address{Department of Mathematics, University of Haifa, 3498838 Haifa, Israel}
\email{tmansour@univ.haifa.ac.il}
\author{g\"{o}khan Y\i ld\i r\i m}
\address{Department of Mathematics, Bilkent University, 06800 Ankara, Turkey}
\email{gokhan.yildirim@bilkent.edu.tr} \subjclass[2010]{05A05, 05A15}
\keywords{Longest increasing subsequence problem, Pattern-avoiding permutations, Chebyshev polynomials, Generating functions}

\maketitle
\thispagestyle{empty}

\section*{abstract}
We study the longest increasing subsequence problem for random permutations avoiding the pattern $312$ and another pattern $\tau$ under the uniform probability distribution. We determine the exact and asymptotic formulas for the average length of the longest increasing subsequences for such permutation classes specifically when the pattern $\tau$ is monotone increasing or decreasing, or any pattern of length four.

\section{Introduction}
The study of longest increasing subsequences for uniformly random permutations is a wonderful example of a research program which begins with an easy-to-state question whose solution makes surprising connections with different branches of mathematics, and culminates with many astonishing results that have interesting applications in statistics, computer science, physics and biology, see \cite{AD, Cor, De1, De2}.
Let $\sigma=\sigma_1\sigma_2\ldots\sigma_n$ be a permutation of
$[n]:=\{1,\ldots, n\}$. We denote by  $L_n(\sigma)$ the length of a
\textit{longest increasing subsequence} in $\sigma$, that
is,
$$L_n(\sigma)=\max\{k\in[n]:\hbox{there exist }1\leq
i_1<i_2<\cdots<i_k\leq n \hbox{ and }
\sigma_{i_1}<\sigma_{i_2}<\cdots<\sigma_{i_k}\}.$$
Note that, in general, such a subsequence is not unique. Erd\"os-Szekeres theorem \cite{ES}
states that every permutation of length $n\geq (r-1)(s-1)+1$
contains either an increasing subsequence of length $r$ or a
decreasing subsequence of length $s$. After this celebrated
result, many researchers worked on the problem of determining
the asymptotic behavior of $L_n$ on $S_n$, the set of all permutations of length $n$, under the \textit{uniform} probability distribution \cite{Hamm, LoSh, Sc, Ulam, VK}. The problem has been studied by several distinct methods from probability theory, random matrix theory, representation theory and statistical mechanics, see \cite{AD, BDS, Joh, Rom} and references therein. It is finally known that $\E(L_n)\sim 2\sqrt n$ \cite{LoSh, Sep, VK} and $n^{-1/6}(L_n-\E(L_n))$ converges in distribution to the Tracy-Widom distribution as $n\to \infty$ \cite{BDJ, TW}. For a thorough exposition of the subject, see the books \cite{BDS, Rom}.

We shall study the longest increasing subsequence problem for some
pattern-avoiding permutation classes. First, we shall recall some
definitions. For permutations $\tau\in S_k$ and $\sigma\in S_n$,
we say that $\tau$ appears as a \textit{pattern} in $\sigma$ if
there is a subsequence
$\sigma_{i_1}\sigma_{i_2}\cdots\sigma_{i_k}$ of length $k$ in
$\sigma$ which has the same relative order of $\tau$, that is,
$\sigma_{i_s}<\sigma_{i_t}$ if and only if $\tau_s<\tau_t$ for all
$1\leq s,t\leq k$. For example, the permutation $312$ appears as a pattern in $52341$ because it has the subsequences $523--$, $52-4-$ or $5-34-$. If $\tau$ does not appear as a pattern in
$\sigma$, then $\sigma$ is called a $\tau$-\textit{avoiding}
permutation. We denote by $S_n(\tau)$ the set of permutations
of length $n$ that avoid the pattern $\tau$. More generally, for
a set $T$ of patterns, we use the notation $S_n(T)=\cap_{\tau\in
T}S_n(\tau)$. Pattern-avoiding permutations have been studied from
combinatorics perspectives for many years, for an introduction to the subject, see \cite{Bon, Vatter}. Recently probabilistic
study of random pattern-avoiding permutations has also been
initiated, and many interesting results have already appeared \cite{BBFGP, HRS1, HRS2, HRS3, Jan, MP, MY, MRS}.

The longest increasing subsequence problem for the pattern-avoiding permutations was first studied for the patterns of length tree, that is, for $\tau\in S_3$ on $S_n(\tau)$ with uniform probability distribution in \cite{DHW}. The case $S_n(\tau^1,\tau^2)$ with $\tau^1,\tau^2 \in S_3$ is studied for all possible cases in \cite{MY}. One of the corollaries of our main result, Theorem~\ref{mainthm}, covers the case $S_n(312,\tau)$ with either $\tau\in S_3(312)$ or $\tau\in S_4(312)$ and hence add some new results to this research program.

Note that for any $\sigma\in S_n$, we have
\begin{equation}
L_n(\sigma)=L_n(\sigma^{rc})=L_n(\sigma^{-1})
\end{equation} where the \textit{reverse, complement} and \textit{inverse} of $\sigma$ are defined as $\sigma^r_i=\sigma_{n+1-i}$, $\sigma^c_i=n+1-\sigma_i$ and $\sigma^{-1}_i=j$ if and only if $\sigma_j=i$, respectively. These symmetries significantly reduce the number of cases needed to be studied.

In a different direction of research, the longest increasing subsequence problem has also been studied on $S_n$ under some \textit{non-uniform} probability distributions such as Mallow distribution \cite {BP, BB, MS} and in some other context such as colored permutations \cite{Bor}, independent-identically distributed sequences, and random walks \cite{ABP, DZ1}.

The paper is organized as follows. We present our main result, Theorem~\ref{mainthm}, in Section~\ref{general} and as a first case apply it to $S_n(312,\tau)$ with $\tau\in S_3(312)$ which gives an alternative proof for some cases considered in \cite{MY} through generating functions. In Section~\ref{secspecial}, we consider three specific longer patterns where $\tau$ is the monotone increasing/decreasing pattern or the pattern $(m-1)m(m-2)(m-3)\cdots321$. The last section presents the results for the case $S_n(312,\tau)$ with $\tau\in S_4(312)$.

For the rest of the paper, we only deal with random variables defined on sets $S_n(312,\tau)$ under the uniform probability distribution. That is,  for any subset $A\subset S_n(312,\tau)$, $\bP^{\tau}(A)=\frac{|A|}{|S_n(312,\tau)|}$. The notation $\E^{\tau}(X)$ denotes the expected value of a random variable $X$ on $S_n(312,\tau)$ under $\bP^{\tau}.$ We denote the coefficient of $x^n$ in a generating function $G(x)$ by $[x^n]G$. For two sequences $\{a_n\}_{n\geq 1}$ and $\{b_n\}_{n\geq 1}$, we write $a_n\sim b_n$ if $\lim_{n\rightarrow\infty}\frac{a_n}{b_n}=1$.


\section{General Results}\label{general}
Note that if $\tau \notin S_k(312)$, then $S_n(312,\tau)=S_n(312)$ for all $n\geq 1$.
For any $\tau\in S_k(312)$ with $k\geq 2$, we define the generating function
\begin{equation}\label{genfunc}
F_{\tau}(x,q)=\sum_{n\geq 0}\sum_{\sigma \in
S_n(312,\tau)}x^nq^{L_n(\sigma)}
\end{equation}
with $F_1(x,q)\equiv1$.

We will use the following facts repeatedly in our proofs:
$$\E^{\tau}(L_n)=\frac{1}{{s_n}}[x^n]\frac{\partial}{\partial q} F_{\tau}(x,q)\Bigr|_{\substack{q=1}} \hbox{ and } \E^{\tau}(L^2_n)=\frac{1}{{s_n}}[x^n]\left(\frac{\partial^2}{\partial q^2} F_{\tau}(x,q)\Bigr|_{\substack{q=1}}+\frac{\partial}{\partial q} F_{\tau}(x,q)\Bigr|_{\substack{q=1}} \right)$$
where $s_n=|S_n(312,\tau)|$. Note also that $s_n=[x^n]F_{\tau}(x,1)$.

For any sequence $w= w_1w_2\cdots w_m$ of $m$-distinct integers, we define the corresponding {\em reduced form} to be the unique permutation $v=v_1v_2\cdots v_m$ where $v_i=\ell$ if the $w_i$ is the $\ell$-th smallest term in $w$. For example, the reduced form of $253$ is $132$. For any sequence $w$, we define $F_w(x,q)$ to be $F_v(x,q)$ where $v$ is the {\it reduced form} of $w$.

In order to determine $F_\tau(x,q)$ explicitly, we shall introduce some notations.
Let $w^1, w^2$ be two sequences of integers, we write $w^1<w^2$ or $w^2>w^1$ if $w^1_i<w^2_j$ for all possible $i,j$. Recall that for any permutation $\tau=\tau_1\cdots\tau_k$, $\tau_i$ is called a {\em right-to-left minimum} if $\tau_i<\tau_j$ for all $j>i$. Note that, by definition, the last entry $\tau_k$ is a right-to-left minimum. Let $\tau\in S_k(312)$ and let $m_0=1<m_1<\ldots<m_r$ be the right-to-left minima of $\tau$ written from left to right. Then $\tau$ can be represented as
$$\tau=\tau^{(0)}m_0\tau^{(1)}m_1\cdots\tau^{(r)}m_r,$$
where $m_0<\tau^{(0)}<m_1<\tau^{(1)}<\cdots<m_r<\tau^{(r)}$, and $\tau^{(j)}$ (may possibly be empty) avoids $312$ for each $j=0,1,\ldots,r$. In this case we call this representation the {\em normal form} of $\tau$. For instance, if $\tau=214365$, then the normal form of $\tau$ is $\tau^{(0)}1\tau^{(1)}3\tau^{(2)}5$ with $\tau^{(0)}=2$, $\tau^{(1)}=4$ and $\tau^{(2)}=6$.

Assume that $\tau\in S_k(312)$ is given in its normal form, that is, $\tau=\tau^{(0)}m_0\tau^{(1)}m_1\cdots\tau^{(r)}m_r$. We use $\Theta^{(j)}$ to denote $\tau^{(0)}m_0\tau^{(1)}m_1\cdots\tau^{(j)}m_j$, which we call the $j^{th}$ {\em prefix} of $\tau$. We use $\Theta^{<j>}$ to denote the reduced form of $\tau^{(j)}m_j\tau^{(j+1)}m_{j+1}\cdots\tau^{(r)}m_r$, which we call the $j^{th}$ {\em suffix} of $\tau$. We set $\Theta^{(-1)}=\emptyset$.

The following lemma plays a key role in the proof of Theorem~\ref{mainthm}. For the sake of the reader, we provide a proof for it which is indeed very similar to the main result of \cite{MV1}.

\begin{lemma}\label{lemma1}
Assume that $\tau\in S_k(312)$ is given in its normal form,
$\tau=\tau^{(0)}m_0\tau^{(1)}m_1\cdots\tau^{(r)}m_r$. Then
$\sigma=\sigma'1\sigma''$ avoids both $312$ and $\tau$ if and only
if there exists $i$, $0\leq i\leq r$, such that $\sigma'1$ avoids
$\Theta^{(i)}$ and contains $\Theta^{(i-1)}$, while $\sigma''$
avoids $\Theta^{<i>}$.
\end{lemma}

\begin{proof}
We denote the set of all permutations, including the empty
permutation, that avoid both $312$ and $\tau$ by $T_\tau$, and that avoid
both $312$ and $\tau$ and contain $\tau'$ by $T_{\tau;\tau'}$. Let $\sigma$ be any nonempty
permutation in $T_\tau$. We can write $\sigma$ as
$\sigma=\sigma'1\sigma''$. Note that $\sigma$ avoids $312$ if and only
if $\sigma'<\sigma''$ and both $\sigma'$ and $\sigma''$ avoid
$312$. Note also that $\Theta^{(j)}$ is a prefix of $\Theta^{(j+1)}$
and $\Theta^{<j+1>}$ is a suffix of $\Theta^{<j>}$ for all
$j=0,1,\ldots,r-1$, and $\Theta^{(r)}=\Theta^{<0>}=\tau$. We have
\begin{align*}
T_\tau&=T_{\Theta^{(0)}}\cup T_{\tau;\Theta^{(0)}},\\
T_{\tau;\Theta^{(s)}}&=T_{\Theta^{(s+1)};\Theta^{(s)}}\cup
T_{\tau;\Theta^{(s+1)}},\qquad s=0,1,2,\ldots,r-1,
\end{align*}
with $T_{\tau;\Theta^{(r)}}=T_{\tau;\tau}=\emptyset$.  Therefore,
\begin{align*}
T_\tau =T_{\Theta^{(0)}}\cup T_{\Theta^{(1)};\Theta^{(0)}}\cup
 T_{\Theta^{(2)};\Theta^{(1)}}\cup\cdots\cup
 T_{\Theta^{(r)};\Theta^{(r-1)}}.
\end{align*}
Thus, since $\sigma\in T_\tau$ we have that $\sigma' 1\in T_\tau$,
so there exists $i$, $0\leq i\leq r$, such that $\sigma' 1\in
T_{\Theta^{(i)};\Theta^{(i-1)}}$. But then we must have $\sigma''\in
T_{\Theta^{<i>}}$. Hence, $\sigma\in T_\tau$ implies that

$\qquad (*)\qquad$ there exists $i$, $0\leq i\leq r$, such that
$\sigma'1\in T_{\Theta^{(i)};\Theta^{(i-1)}}$ and $\sigma''\in
T_{\Theta^{<i>}}$.

Note that in the case $i=0$, $\sigma'1$ avoids
$\Theta^{(0)}$, while $\sigma''$ avoids $\Theta^{<0>}$, and we
defined $\Theta^{(-1)}=\emptyset$. Then clearly $\sigma'1$
contains $\Theta^{(-1)}$.

On the other hand, let both $\sigma'1$ and $\sigma''$ avoid $312$ and
satisfy the condition $(*)$, that is, there exists $i$, $0\leq i\leq
r$, such that $\sigma'1\in T_{\Theta^{(i)};\Theta^{(i-1)}}$ and
$\sigma'' \in T_{\Theta^{<i>}}$. Thus, both $\sigma'1$ and $\sigma''$ avoid
$\tau$. If $\sigma=\sigma'1\sigma''$ contains $\tau$, then
$\sigma'1$ contains $\Theta^{(j-1)}$ and $\sigma''$ contains
$\Theta^{<j>}$, for some $j$, $0\leq j\leq r$. If we choose $j$ to
be maximal (it exists since $\sigma'1$ avoids $\Theta^{(r)}=\tau$
and $\Theta^{(j)}$ is a prefix of $\Theta^{(j+1)}$), then we see
that $\sigma'1$ avoids $\Theta^{(j)}$ and contains $\Theta^{(j-1)}$
while $\sigma''$ contains $\Theta^{<i>}$, which contradicts $(*)$.
Thus the condition $(*)$ implies that $\sigma\in T_\tau$. This
completes the proof.
\end{proof}

Note that in the above lemma $\sigma'1$ avoids $\Theta^{(i)}$ and contains $\Theta^{(i-1)}$ if and only if $\sigma'$ avoids $\Theta^{(i)}$ and contains $\Theta^{(i-1)}$, for all $2\leq i\leq r$.

Our main result gives a functional equation for the generating
function.

\begin{theorem}\label{mainthm}
Let $\tau\in S_k(312)$ be given in its normal form $\tau^{(0)}m_0\tau^{(1)}m_1\cdots\tau^{(r)}m_r$ with $k\geq2$. 
\begin{itemize}
\item If $\tau^{(0)}=\emptyset$, then
\begin{align*}
F_\tau(x,q)&=1+xq+x(F_{\tau}(x,q)-1)+xq(F_{\Theta^{<1>}}(x,q)-1)\\
&+x\sum_{j=1}^{r}(F_{\Theta^{(j)}}(x,q)-F_{\Theta^{(j-1)}}(x,q))(F_{\Theta^{<j>}}(x,q)-1);
\end{align*}
\item if $\tau^{(0)}\neq\emptyset$, then
\begin{align*}
F_\tau(x,q)&=1+xq+x(F_{\tau^{(0)}}(x,q)-1)\delta_{r=0}+x(F_{\tau}(x,q)-1)\delta_{r\geq1}+xq(F_\tau(x,q)-1)\\
&+x\sum_{j=2}^{r}(F_{\Theta^{(j)}}(x,q)-F_{\Theta^{(j-1)}}(x,q))(F_{\Theta^{<j>}}(x,q)-1)\\
&+x(F_{\Theta^{(1)}}(x,q)-F_{\tau^{(0)}}(x,q))(F_{\Theta^{<1>}}(x,q)-1)\delta_{r\geq1}\\
&+x(F_{\tau^{(0)}}(x,q)-1)(F_{\tau}(x,q)-1),
\end{align*}
\end{itemize}
where we define $F_\emptyset(x,q)=0$, and $\delta_\chi$ denotes $1$ if the condition $\chi$ holds, and $0$ otherwise.
\end{theorem}
\begin{proof}
Note that the contributions of the empty permutation and the
permutation of length $1$ to the generating function are $1$ and
$xq$, respectively. Henceforth, we assume that $\tau$ has at least two entries. Let $\sigma=\sigma'1\sigma''$ be any nonempty permutation which avoids both $312$ and $\tau$. 

{\bf Case $\tau^{(0)}=\emptyset$}: Since $\tau$ has at least two entries and $\tau^{(0)}=\emptyset$, we have that $r\geq1$. If $\sigma=\sigma'1$ with $\sigma'\neq\emptyset$, then we have the contribution of $x(F_\tau(x,q)-1)$. If $\sigma=1\sigma''$ with $\sigma''\neq\emptyset$, then we have the
contribution of $xq(F_{\Theta^{<1>}}(x,q)-1)$. As a last
case, we need to consider the permutations in the form of
$\sigma=\sigma'1\sigma''$ with $\sigma',\sigma''\neq\emptyset$. By Lemma \ref{lemma1}, the contribution of this case is given by
\begin{align*}
&x\sum_{j=2}^{r}(F_{\Theta^{(j)}}(x,q)-F_{\Theta^{(j-1)}}(x,q))(F_{\Theta^{<j>}}(x,q)-1)+x(F_{\Theta^{(1)}}(x,q)-1)(F_{\Theta^{<1>}}(x,q)-1)\\
&=x\sum_{j=1}^{r}(F_{\Theta^{(j)}}(x,q)-F_{\Theta^{(j-1)}}(x,q))(F_{\Theta^{<j>}}(x,q)-1),
\end{align*}
where in last equality we used that $\Theta^{(0)}=1$ and $F_{\Theta^{(0)}}(x,q)=1$.
By summing over all the contributions, we complete
the proof.

{\bf Case $\tau^{(0)}\neq\emptyset$}:
If $\sigma=\sigma'1$ with $\sigma'\neq\emptyset$, then we have the
contribution of $x(F_\tau(x,q)-1)$ when $r\geq1$, and
$x(F_{\tau^{(0)}}(x,q)-1)$ when $r=0$. If $\sigma=1\sigma''$ with $\sigma''\neq\emptyset$, then we have the
contribution of $xq(F_{\tau}(x,q)-1)$. As a last
case, we need to consider the permutations in the form of
$\sigma=\sigma'1\sigma''$ with $\sigma',\sigma''\neq\emptyset$. By
Lemma~\ref{lemma1}, the contribution of this case is given by
\begin{align*}
&x\sum_{j=2}^{r}(F_{\Theta^{(j)}}(x,q)-F_{\Theta^{(j-1)}}(x,q))(F_{\Theta^{<j>}}(x,q)-1)\\
&\qquad+x(F_{\Theta^{(1)}}(x,q)-F_{\tau^{(0)}}(x,q))(F_{\Theta^{<1>}}(x,q)-1)\delta_{r\geq1}+x(F_{\tau^{(0)}}(x,q)-1)(F_{\tau}(x,q)-1),
\end{align*}

 By summing over all the contributions, we complete
the proof.
\end{proof}

We can also deduce the rationality of the generating function $F_\tau(x,q)$ for any nonempty pattern $\tau$ by using the induction on $k$ with the observations in the proof of Theorem \ref{mainthm} and $F_1(x,q)=1$.

\begin{corollary}\label{mainthmrat}
For any $k\geq 1$ and $\tau\in S_k(312)$, the generating function $F_\tau(x,q)$ is a rational function in $x$ and $q$.
\end{corollary}

Note that $F_1(x,q)=1$ (the only permutation that avoids the pattern $1$ is the empty permutation). Theorem \ref{mainthm} with $\tau=21$ gives
$$F_{21}(x,q)=1+xq+x(F_1(x,q)-1)+xq(F_{21}(x,q)-1)+x(F_1(x,q)-1)(F_{21}(x,q)-1),$$
where $F_1(x,q)=1$. Thus, $F_{21}(x,q)=\frac{1}{1-xq}$. Theorem \ref{mainthm} with $\tau=12$ gives
\begin{align*}
F_{12}(x,q)&=1+xq+x(F_{12}(x,q)-1)+xq(F_1(x,q)-1)+x(F_{1}(x,q)-1)(F_{12}(x,q)-1).
\end{align*}
Thus, $F_{12}(x,q)=\frac{1+xq-x}{1-x}=1+\frac{xq}{1-x}$. We summarize these results in the following corollary for future references.

\begin{corollary}\label{pat2s}
For $\tau \in S_2$, the generating functions are given by
$$F_{21}(x,q)=\frac{1}{1-xq}\mbox{ and }F_{12}(x,q)=1+\frac{xq}{1-x}.$$
\end{corollary}

Next we will consider the application of Theorem \ref{mainthm} to the patterns of length three, that is, $\tau\in S_3(312)$. In each case we will also use Corollary \ref{pat2s} and $F_1(x,q)=1$.

{\bf Pattern $\bf \tau=123$}. We have $\Theta^{(0)}=1$, $\Theta^{(1)}=12$, $\Theta^{(2)}=123$,  and $\Theta^{<0>}=123$, $\Theta^{<1>}=12$, $\Theta^{<2>}=1$. Thus,
$$F_{123}(x,q)=1+xq+x(F_{123}(x,q)-1)+xq(F_{12}(x,q)-1)+x(F_{12}(x,q)-1)(F_{12}(x,q)-1),$$
which gives
$F_{123}(x,q)=1+xq/(1-x)+\frac{x^2q^2}{(1-x)^3}$.

{\bf Pattern $\bf \tau=132$}. We have $\Theta^{(0)}=1$, $\Theta^{(1)}=132$, and $\Theta^{<0>}=132$, $\Theta^{<1>}=21$. Thus,
$$F_{132}(x,q)=1+xq+x(F_{132}(x,q)-1)+xq(F_{21}(x,q)-1)+x(F_{132}(x,q)-1)(F_{21}(x,q)-1),$$
which gives
$F_{132}(x,q)=\frac{1-x}{1-x-xq}$.

{\bf Pattern $\bf \tau=213$}. We have $\Theta^{(0)}=21$, $\Theta^{(1)}=213$, and $\Theta^{<0>}=213$, $\Theta^{<1>}=1$. Thus,
$$F_{213}(x,q)=1+xq+x(F_{213}(x,q)-1)+xq(F_{213}(x,q)-1),$$
which gives
$F_{213}(x,q)=\frac{1-x}{1-x-xq}$.

{\bf Pattern $\bf \tau=231$}. We have $\Theta^{(0)}=\Theta^{<0>}=231$. Thus,
$$F_{231}(x,q)=1+xq+x(F_{12}(x,q)-1)+xq(F_{231}(x,q)-1)+x(F_{12}(x,q)-1)(F_{231}(x,q)-1),$$
which gives
$F_{231}(x,q)=\frac{1-x}{1-x-xq}$.

{\bf Pattern $\bf \tau=321$}. We have $\Theta^{(0)}=\Theta^{<0>}=321$. Thus,
$$F_{321}(x,q)=1+xq+x(F_{21}(x,q)-1)+xq(F_{321}(x,q)-1)+x(F_{21}(x,q)-1)(F_{321}(x,q)-1),$$
which gives
$F_{321}(x,q)=\frac{1-xq}{(1-xq)^2-x^2q}$.

Hence, we can state the following result.
\begin{corollary}\label{pat3s} For $\tau\in S_3(312)$, the generating functions are given by
\begin{align*}
F_{123}(x,q)&=1+\frac{xq}{1-x}+\frac{x^2q^2}{(1-x)^3},\\
F_{132}(x,q)&=F_{213}(x,q)=F_{231}(x,q)=\frac{1-x}{1-x-xq},\\
F_{321}(x,q)&=\frac{1-xq}{(1-xq)^2-x^2q}.
\end{align*}
\end{corollary}

The results in Corollary~\ref{pat3s} indeed extend the relevant results of Simion and Schmidt \cite{SmSc} for the permutations avoiding two patterns of length three. Here we find the generating functions for the number of permutations $\sigma$ in $S_n(312,\tau)$ with $\tau\in S_3(312)$ according to the length of the longest increasing subsequence in $\sigma$.

Our next result considers a specific type of pattern in which the last entry is $1$.
\begin{corollary}\label{cor2} Assume $\tau=\rho1\in S_k(312)$ and $k\geq2$. Then $F_\tau(x,q)=\frac{1}{1-xq-x(F_{\rho}(x,q)-1)}$.
Moreover,
$$\frac{\partial}{\partial q} F_{\tau}(x,q)\Bigr|_{\substack{q=1}}=xF^2_{\tau}(x,1)\left(1+\frac{\partial}{\partial q}F_{\rho}(x,q)\Bigr|_{\substack{q=1}}\right).$$
\end{corollary}
\begin{proof}
By Theorem \ref{mainthm} with $\tau=\rho1$ ($r=0$, $m_0=1$ and $\tau^{(0)}=\rho$), we have
$$F_{\tau}(x,q)=1+xq+x(F_{\rho}(x,q)-1)+xq(F_{\tau}(x,q)-1)+x(F_{\rho}(x,q)-1)(F_{\tau}(x,q)-1),$$
which implies
$$F_{\tau}(x,q)=\frac{1}{1-xq-x(F_{\rho}(x,q)-1)}.$$
In particular, $F_{\tau}(x,1)=\frac{1}{1-xF_{\rho(x,1)}}$, as shown in \cite{MV1}. Moreover, by differentiating $F_{\tau}(x,q)$ with respect to $q$ and evaluating at $q=1$, we obtain
$$\frac{\partial}{\partial q} F_{\tau}(x,q)\Bigr|_{\substack{q=1}}=\frac{x\left(1+\frac{\partial}{\partial q} F_{\rho}(x,q)\Bigr|_{\substack{q=1}}\right)}{(1-xF_{\rho}(x,1))^2}
=xF_\tau^2(x,1)\left(1+\frac{\partial}{\partial q} F_{\rho}(x,q)\Bigr|_{\substack{q=1}}\right),$$
which completes the proof.
\end{proof}

 By Corollary \ref{pat2s} and \ref{pat3s}, we recover the relevant results in \cite{MY}.
\begin{theorem}\label{pat3sEE} For all $n\geq1$, we have
\begin{align*}
\E^{123}(L_n)&=\frac{2(n^2-n+1)}{n^2-n+2},&&\E^{123}(L_n^2)=\frac{2(2n^2-2n+1)}{n^2-n+2},\\
\E^{132}(L_n)&=\E^{213}(L_n)=\E^{231}(L_n)=\frac{n+1}{2},&&\E^{132}(L_n^2)=\E^{213}(L_n^2)=\E^{231}(L_n^2)=\frac{n(n+3)}{4},\\
\E^{321}(L_n)&=\frac{3n}{4},&&\E^{321}(L_n^2)=\frac{n(9n+1)}{16}.
\end{align*}
\end{theorem}

\section{Special cases of longer patterns}\label{secspecial}
The main result of this paper, Theorem~\ref{mainthm},  can be used to obtain general results for several longer patterns. In this subsection, as an example, we apply it to the following three specific patterns $12\cdots m$, $m(m-1)\cdots21$ and $(m-1)m(m-2)\cdots21$.

Recall that the Chebyshev polynomials of the second kind are defined by $U_j(\cos\theta)= \frac{\sin((j +1)\theta)}{\sin\theta}$. It is well known that these polynomials satisfy
\begin{align}\label{eqch}
U_0(t)=1, U_1(t)=2t,\mbox{ and }U_m(t)=2tU_{m-1}(t)-U_{m-2}(t)\mbox{ for all integers $m$},
\end{align}
and
\begin{align}\label{eqchf}
U_n(t)=2^n\prod_{j=1}^n\left(t-\cos\left(\frac{j\pi}{n+1}\right)\right).
\end{align}

\subsection{Monotone increasing pattern $\tau=12\cdots m$}
In this subsection, we study the pattern $\tau=12\cdots m$.
By Corollaries \ref{pat2s} and \ref{pat3s}, we see that $F_{1}(x,q)=1$, $F_{12}(x,q)=1+\frac{xq}{1-x}$ and $F_{123}(x,q)=1+\frac{xq}{1-x}+\frac{x^2q^2}{(1-x)^3}$.
By Theorem \ref{mainthm} with $\tau=12\cdots m$, we have
\begin{align*}
F_{12\cdots m}(x,q)&=1+xq+x(F_{12\cdots m}(x,q)-1)+xq(F_{12\cdots(m-1)}(x,q)-1)\\
&+x\sum_{j=2}^{m}(F_{12\cdots j}(x,q)-F_{12\cdots(j-1)}(x,q))(F_{12\cdots(m-j+1)}(x,q)-1).
\end{align*}
which is equivalent to
\begin{align*}
F_{12\cdots m}(x,q)&=1+xqF_{12\cdots(m-1)}(x,q)\\
&+x\sum_{j=2}^{m}(F_{12\cdots j}(x,q)-F_{12\cdots(j-1)}(x,q))F_{12\cdots(m-j+1)}(x,q).
\end{align*}
Define $G(x,q,v)=\sum_{m\geq1}F_{12\cdots m}(x,q)v^m$. Then, by multiplying the above recurrence by $v^m$ and summing over $m$, we obtain
\begin{align*}
\sum_{m\geq1}F_{12\cdots m}(x,q)v^m&=\sum_{m\geq1}v^m+xqv\sum_{m\geq1}v^{m-1}F_{12\cdots(m-1)}(x,q)\\
&+x\sum_{m\geq2}v^m\sum_{j=2}^{m}(F_{12\cdots j}(x,q)-F_{12\cdots(j-1)}(x,q))F_{12\cdots(m-j+1)}(x,q),
\end{align*}
which implies
\begin{align*}
G(x,q,v)&=\frac{v}{1-v}+xqvG(x,q,v)+\frac{x}{v}(G(x,q,v)-v)G(x,q,v)-x(G(x,q,v))^2.
\end{align*}
Thus, $G(x,q,v)$ satisfies
\begin{align*}
\frac{v}{1-v}+(1+x-qxv)G(x,q,v)-\frac{x(1-v)}{v}G^2(x,q,v)=0.
\end{align*}
By solving the above equation for $G(x,q,v)$ we obtain
$$G(x,q,v)=\frac{(1+x-qxv-\sqrt{(1+x-qxv)^2-4x})v}{2x(1-v)}.$$
Then
\begin{align*}
\frac{G(x,q,\frac{v(1-x)^2}{qx})\frac{1-\frac{v(1-x)^2}{qx}}{\frac{v(1-x)^2}{qx}}-1}{(1-x)v}&=\frac{1-v(1-x)-\sqrt{1-2v(1+x)+v^2(1-x)^2}}{2xv(1-v)},
\end{align*}
which, by definition of Narayana numbers (see Sequence A001263 in \cite{Slo}),  leads to
\begin{align*}
\frac{G(x,q,\frac{v(1-x)^2}{qx})\frac{1-\frac{v(1-x)^2}{qx}}{\frac{v(1-x)^2}{qx}}-1}{(1-x)v}
=1+\sum_{n\geq1}\sum_{k=1}^n\frac{1}{n}\binom{n}{k}\binom{n}{k-1}x^{k-1}v^n.
\end{align*}
Therefore, by replacing $v$ by $xqv/(1-x)^2$, we have
\begin{align*}
\frac{G(x,q,v)\frac{1-v}{v}-1}{\frac{xqv}{1-x}}
=1+\sum_{n\geq1}\sum_{k=1}^n\frac{1}{n}\binom{n}{k}\binom{n}{k-1}\frac{x^{n+k-1}}{(1-x)^{2n}}q^nv^n.
\end{align*}
which implies
\begin{align*}
G(x,q,v)=\frac{v}{1-v}\left(1+\frac{qxv}{1-x}+\sum_{n\geq1}\sum_{k=1}^n\frac{1}{n}\binom{n}{k}\binom{n}{k-1}\frac{q^{n+1}x^{n+k}v^{n+1}}{(1-x)^{2n+1}}\right).
\end{align*}
By finding the coefficient of $v^m$, we obtain the following result.
\begin{corollary}\label{thinc}
For all $m\geq1$,
$$F_{12\cdots m}(x,q)=1+\frac{qx}{1-x}+\sum_{j=2}^{m-1}\left(\frac{q^jx^j}{(1-x)^{2j-1}}
\sum_{k=1}^{j-1}\frac{1}{j-1}\binom{j-1}{k}\binom{j-1}{k-1}x^{k-1}\right).$$
\end{corollary}

By Corollary \ref{thinc}, we see that the generating function $F_{12\cdots m}(x,1)$ has a pole at $x=1$ of order $2m-3$. Thus,
$$[x^n]F_{12\cdots m}(x,1)\sim \frac{n^{2m-4}}{(2m-4)!(m-2)}
\sum_{k=1}^{m-2}\binom{m-2}{k}\binom{m-2}{k-1},$$
which, by definition of Narayana numbers, implies that
$$[x^n]F_{12\cdots m}(x,1)\sim \frac{n^{2m-4}}{(2m-4)!}c_{m-2},$$
where $c_n=\frac{1}{n+1}\binom{2n}{n}$ is the $n^{th}$ Catalan number.

Also, by Corollary \ref{thinc}, we see that the generating function $\frac{\partial}{\partial q}F_{12\cdots m}(x,q)\mid_{q=1}$ has a pole at $x=1$ of order $2m-3$. Thus,
$$[x^n]\frac{\partial}{\partial q}F_{12\cdots m}(x,q)\mid_{q=1}\sim \frac{(m-1)n^{2m-4}}{(2m-4)!(m-2)}
\sum_{k=1}^{m-2}\binom{m-2}{k}\binom{m-2}{k-1}=\frac{(m-1)n^{2m-4}}{(2m-4)!}c_{m-2}.$$
Hence, we can state the following result.
\begin{theorem}\label{thmm21}
Let $m\geq1$. When $n\rightarrow\infty$, we have
\begin{align*}
\E^{12\cdots m}(L_n)&\sim m-1.
\end{align*}
\end{theorem}

\subsection{Monotone decreasing pattern $\tau=m(m-1)\cdots21$}
In this subsection, we study the pattern $\mm=m(m-1)\cdots21$.
\begin{corollary}\label{cor1}Let $\mm=m(m-1)\cdots21$. Then
$$F_{\mm}(x,q)=\frac{U_{m-2}(t)-\sqrt xU_{m-3}(t)}{\sqrt x(U_{m-1}(t)-\sqrt xU_{m-2}(t))},$$
where $t=\frac{1+x-xq}{2\sqrt x}$.
\end{corollary}
\begin{proof}
The proof is given by induction on $m$. Clearly, $F_1(x,q)=1$ and $F_{21}(x,q)=\frac{1}{1-xq}$, so the claim holds for $m=1,2$.
Assume that the claim holds for $1,2,\cdots, m$ and let's us prove it for $m+1$.  Since $\mm+{\bf1}=(\mm +1)1$, then Corollary \ref{cor2} gives $F_{\mm+{\bf1}}(x,q)=\frac{1}{1-xq-x(F_{\mm}(x,q)-1)}$. Thus by induction assumption, we obtain
\begin{align*}
F_{\mm+{\bf1}}(x,q)&=\frac{1}{1-xq-x(F_{\mm}(x,q)-1)}\\
&=\frac{\sqrt x(U_{m-1}(t)-\sqrt xU_{m-2}(t))}{x(2tU_{m-1}(t)-U_{m-2}(t))-x\sqrt x(2tU_{m-2}(t)-U_{m-3}(t))}\\
&=\frac{\sqrt x(U_{m-1}(t)-\sqrt xU_{m-2}(t))}{x(U_m(t)-\sqrt xU_{m-1}(t))}\\
&=\frac{U_{m-1}(t)-\sqrt xU_{m-2}(t)}{\sqrt x(U_m(t)-\sqrt xU_{m-1}(t))}
\end{align*}
where we used the fact \eqref{eqch} and $2t\sqrt x=1+x-xq$.
\end{proof}

By Corollary \ref{cor1} with $q=1$ and \eqref{eqch}, we have
$F_{\mm}(x,1)=\frac{U_{m-1}(\frac{1}{2\sqrt{x}})}{\sqrt{x}U_{m}(\frac{1}{2\sqrt{x}})}$, as
shown in  \cite{MV1}.
Moreover, by Corollary \ref{cor2}, we have
$$\frac{\partial}{\partial q} F_{\mm}(x,q)\Bigr|_{\substack{q=1}}=xF^2_{\mm}(x,1)\left(1+\frac{\partial}{\partial q}F_{{\bf m-1}}(x,q)\Bigr|_{\substack{q=1}}\right)$$
with $\frac{\partial}{\partial q} F_{{\bf1}}(x,q)\Bigr|_{\substack{q=1}}=0$.
Thus, by induction on $m$, we can state the following result.

\begin{corollary}\label{cor1a}Let $\mm=m(m-1)\cdots21$. Then
$$\frac{\partial}{\partial q} F_{\mm}(x,q)\Bigr|_{\substack{q=1}}=\frac{1}{U_m^2(\frac{1}{2\sqrt{x}})}\sum_{j=1}^{m-1}U_j^2(\frac{1}{2\sqrt{x}}).$$
\end{corollary}

Since the smallest
pole of $1/U_n(x)$ is $\cos\left(\frac{\pi}{n+1}\right)$, it
follows, by Corollary \ref{cor1a}, that the coefficient of $x^n$ in the generating function
$\frac{\partial}{\partial q} F_{\mm}(x,q)\Bigr|_{\substack{q=1}}$
is given by,
\begin{align*}
[x^n]\frac{\partial}{\partial
q}F_{\mm}(x,q)\Bigr|_{\substack{q=1}}\sim\alpha_m
n\left(4\cos^2\left(\frac{\pi}{m+1}\right)\right)^n \hbox{ as } n\rightarrow\infty.
\end{align*}
Let $v_0=\frac{1}{4\cos^2\left(\frac{\pi}{m+1}\right)}$. The constant $\alpha_m$ can be computed explicitly as
\begin{align*}
\alpha_m&=\lim_{x\rightarrow v_0}
\frac{\left(1-4x\cos^2\left(\frac{\pi}{m+1}\right)\right)^2}{U_m^2(\frac{1}{2\sqrt{x}})}\sum_{j=1}^{m-1}U_j^2(\frac{1}{2\sqrt{x}})\\
&=\frac{\sum_{j=1}^{m-1}U_j^2\left(\cos\left(\frac{\pi}{m+1}\right)\right)}
{\left(4\cos^2\left(\frac{\pi}{m+1}\right)\right)^m\prod_{j=2}^{m-1}\left(1-\frac{\cos\left(\frac{j\pi}{m+1}\right)}{\cos\left(\frac{\pi}{m+1}\right)}\right)^2}\\
&=\frac{\sum_{j=1}^{m-1}U_j^2\left(\cos\left(\frac{\pi}{m+1}\right)\right)}
{4^m\cos^4\left(\frac{\pi}{m+1}\right)\prod_{j=2}^{m-1}\left(\cos\left(\frac{\pi}{m+1}\right)-\cos\left(\frac{j\pi}{m+1}\right)\right)^2}.
\end{align*}
Moreover, the coefficient of $x^n$ in
the generating function
$F_{\mm}(x,1)=\frac{U_{m-1}(\frac{1}{2\sqrt{x}})}{\sqrt{x}U_m(\frac{1}{2\sqrt{x}})}$
is given by
\begin{align}
[x^n]F_{\mm}(x,1)\sim\tilde \alpha_m
\left(4\cos^2\left(\frac{\pi}{m+1}\right)\right)^n \hbox{ as }
n\rightarrow\infty,\label{eqmma1}
\end{align}
where
\begin{align*}
\tilde \alpha_m&=\lim_{x\rightarrow v_0}
\frac{\left(1-4x\cos^2\left(\frac{\pi}{m+1}\right)\right)U_{m-1}(\frac{1}{2\sqrt{x}})}{\sqrt{x}U_m(\frac{1}{2\sqrt{x}})}\\
&=\frac{U_{m-1}(\cos(\frac{\pi}{m+1}))}{2^{m-1}\cos^{m-1}(\frac{\pi}{m+1})
\prod_{j=2}^{m-1}\left(1-\frac{\cos\left(\frac{j\pi}{m+1}\right)}{\cos\left(\frac{\pi}{m+1}\right)}\right)}\\
&=\frac{U_{m-1}(\cos(\frac{\pi}{m+1}))}
{2^{m-1}\cos\left(\frac{\pi}{m+1}\right)\prod_{j=2}^{m-1}\left(\cos\left(\frac{\pi}{m+1}\right)-\cos\left(\frac{j\pi}{m+1}\right)\right)}.
\end{align*}
Thus we have
$\E^{\mm}(L_n)\sim \frac{\alpha_m}{\tilde \alpha_m}n$.
By substituting expressions of $\alpha_m$ and $\tilde \alpha_m$,
it leads to the following result.
\begin{theorem}\label{thmm21}
Let $m\geq1$. When $n\rightarrow\infty$, we have
\begin{align*}
\E^{\mm}(L_n)&\sim \frac{\sum_{j=1}^{m-1}U_j^2\left(\cos\left(\frac{\pi}{m+1}\right)\right)}
{2^{m+1}\cos^3\left(\frac{\pi}{m+1}\right)U_{m-1}(\cos(\frac{\pi}{m+1}))\prod_{j=2}^{m-1}\left(\cos\left(\frac{\pi}{m+1}\right)-\cos\left(\frac{j\pi}{m+1}\right)\right)}n.
\end{align*}
\end{theorem}

For example, Theorem \ref{thmm21} for $m=3$ gives that
$\E^{321}(L_n)\sim\frac{3n}{4}$ as shown in \cite{MY}, and for $m=4$, we have
$\E^{4321}(L_n)\sim \left(2-\frac{3}{\sqrt{5}}\right)n$.

\subsection{Pattern $\tau=(m-1)m(m-2)(m-3)\cdots321$}
In this subsection, we study the pattern
$\hat\mm=(m-1)m(m-2)(m-3)\cdots321$.
\begin{corollary}\label{cor3a}Let $\hat\mm=(m-1)m(m-2)(m-3)\cdots321$ with $m\ge3$. Then
$$F_{\hat\mm}(x,q)=\frac{(1-x)U_{m-3}(t)-\sqrt x (1-x+xq)U_{m-4}(t)}{\sqrt x ((1-x)U_{m-2}(t)-\sqrt{x}(1-x+xq)U_{m-3}(t))},$$
where $t=\frac{1+x-xq}{2\sqrt x}$.
\end{corollary}
\begin{proof}
We proceed the proof by induction on $m$. By Example $\tau=231$ in
Section 2, we have $F_{\hat{\bf3}}(x,q)=\frac{1-x}{1-x-xq}$, so
the result holds for $m=3$. Assume that the result holds for $m-1$
and let us prove for $m$. By Corollary \ref{cor2}, we have
$$F_{\hat\mm}(x,q)=\frac{1}{1+x-xq-xF_{\hat{\mathbf m-1}}(x,q)}.$$
Thus by induction hypothesis, we have
\begin{align*}
&F_{\hat\mm}(x,q)=\frac{1}{2\sqrt{x}t-\sqrt{x}\frac{(1-x)U_{m-4}(t)-\sqrt
x (1-x+xq)U_{m-5}(t)}{
(1-x)U_{m-3}(t)-\sqrt{x}(1-x+xq)U_{m-4}(t)}}\\
&=\frac{(1-x)U_{m-3}(t)-\sqrt{x}(1-x+xq)U_{m-4}(t)}
{\sqrt{x}\left((1-x)(2tU_{m-3}(t)-U_{m-4}(t))-\sqrt{x}(1-x+xq)(2tU_{m-4}(t)-U_{m-5}(t))\right)},
\end{align*}
which, by \eqref{eqch}, implies
\begin{align*}
F_{\hat\mm}(x,q)=\frac{(1-x)U_{m-3}(t)-\sqrt x
(1-x+xq)U_{m-4}(t)}{\sqrt x
((1-x)U_{m-2}(t)-\sqrt{x}(1-x+xq)U_{m-3}(t))},
\end{align*}
which completes the proof.
\end{proof}

By Corollary \ref{cor3a} with $q=1$ and \eqref{eqch}, we have
$$F_{\hat\mm}(x,1)=\frac{U_{m-1}(\frac{1}{2\sqrt{x}})}{\sqrt{x}U_{m}(\frac{1}{2\sqrt{x}})}.$$
By induction on $m$, we can state the following result.

\begin{corollary}\label{cor3b}Let $\hat\mm=(m-1)m(m-2)(m-3)\cdots321$ with $m\geq4$. Then
$$\frac{\partial}{\partial q} F_{\hat\mm}(x,q)\Bigr|_{\substack{q=1}}=\frac{1}{U_m^2(\frac{1}{2\sqrt{x}})}\left(U_2(\frac{1}{2\sqrt{x}})+\sum_{j=2}^{m-1}U_j^2(\frac{1}{2\sqrt{x}})\right).$$
\end{corollary}

By similar arguments as in the proof of Theorem \ref{thmm21}, we
obtain the following result.
\begin{theorem}\label{thmx1mmx221}
Let $m\geq4$. As $n\rightarrow\infty$, we have
\begin{align*}
\E^{\hat\mm}(L_n)&\sim
\frac{U_2\left(\cos\left(\frac{\pi}{m+1}\right)\right)+\sum_{j=1}^{m-1}U_j^2\left(\cos\left(\frac{\pi}{m+1}\right)\right)}
{2^{m+1}\cos^3\left(\frac{\pi}{m+1}\right)U_{m-1}(\cos(\frac{\pi}{m+1}))\prod_{j=2}^{m-1}\left(\cos\left(\frac{\pi}{m+1}\right)-\cos\left(\frac{j\pi}{m+1}\right)\right)}n.
\end{align*}
\end{theorem}

\section{The case $S_n(312,\tau)$ where $\tau\in S_4(312)$}\label{3-4sec}
In this section, we present the results for random permutations from $S_n(312,\tau)$ where $\tau\in S_4(312)$. A summary of the results for all $\tau\in S_4(312)$ is given in Table~\ref{table4}. We present the details only for the two patterns, $\tau=1234$ and $\tau=1243$. Since the computations for other cases are very similar, the details are omitted.
\begin{table}[!t]
\begin{center}
\begin{tabular}{|| c | c| c|c|c||}\hline
$\tau$ & $F_{\tau}(x,q)$ & $\E^{\tau}(L_n)$ & Reference\\\hline\hline
&&&\\[-9pt]
1234& $1+\frac{xq}{1-x}+\frac{x^2q^2}{(1-x)^3}+\frac{x^3(1+x)q^3}{(1-x)^5}$ & $\frac{3(n^4-4n^3+9n^2-6n+4)}{n^4-4n^3+11n^2-8n+12}\rightarrow3$& Example \ref{ex1234}\\[4pt]\hline
&&&\\[-9pt]
1243,1324&$1+\frac{xq(qx(2x-1)+(1-x)^2)}{(1-x-qx)^2(1-x)}$  & $\frac{2^{n-3}(n^2-n+4)}{(n-1)2^{n-2}+1}\sim\frac{n}{2}$ & Example \ref{ex1243}\\
2134&&& Theorem \ref{mainthm}\\\hline
&&&\\[-9pt]
2314&$1+\frac{xq(1-x)}{(1-x)^2-qx}$& $\sim\frac{n}{\sqrt{5}}$& Theorem \ref{mainthm}\\
1342&&&\\\hline
&&&\\[-9pt]
2143,3214&$\frac{1-x-qx}{(1-qx)^2-x}$&$\sim\frac{n}{\sqrt{5}}$& Theorem \ref{mainthm}\\
2431,3241&&&\\
3421,1432&&&\\\hline
&&&\\[-9pt]
2341,4321&$\frac{(1-x)^3}{(1-x)^3-xq(1-x)^2-x^3q^2}$&$\sim\frac{(-5a^2+22a-9)n}{31}$ &Theorem \ref{mainthm}\\
&&{\tiny$a\approx2.46577\cdots$, $a^3-4a^2+5a-3=0$}&\\\hline
\end{tabular}
\end{center}
\caption{A summary of the results for $S_n(312,\tau)$ with $\tau \in S_4(312)$.}
\label{table4}
\end{table}

\begin{example}\label{ex1234}
By Theorem \ref{mainthm} with $\tau=1234$, we have
\begin{align*}
F_{1234}(x,q)&=1+xqF_{123}(x,q)+x(F_{12}(x,q)-F_{1}(x,q))F_{123}(x,q)\\
&+x(F_{123}(x,q)-F_{12}(x,q))F_{12}(x,q)+x(F_{1234}(x,q)-F_{123}(x,q))F_{1}(x,q).
\end{align*}
By Corollaries \ref{pat2s} and \ref{pat3s}, we have
$$F_{1234}(x,q)=1+\frac{xq}{1-x}+\frac{x^2q^2}{(1-x)^3}+\frac{x^3(1+x)q^3}{(1-x)^5},$$
which agrees with Theorem \ref{thinc} with $m=4$.
Therefore, we have
$$\E^{1234}(L_n)=\frac{3(n^4-4n^3+9n^2-6n+4)}{n^4-4n^3+11n^2-8n+12}$$
and
$$\E^{1234}(L_n^2)=\frac{3(3n^4-12n^3+23n^2-14n+4)}{n^4-4n^3+11n^2-8n+12}.$$
\end{example}

\begin{example}\label{ex1243}
By Theorem \ref{mainthm} with $\tau=1243$, we have
\begin{align*}
F_{1243}(x,q)&=1+xqF_{132}(x,q)+x(F_{1243}(x,q)-1)\\
&+x(F_{12}(x,q)-1)(F_{132}(x,q)-1)+x(F_{1243}(x,q)-F_{12}(x,q))(F_{21}(x,q)-1),
\end{align*}
which, by Corollaries \ref{pat2s} and \ref{pat3s}, leads to
$$F_{1243}(x,q)=1+\frac{xq(qx(2x-1)+(1-x)^2)}{(1-x-qx)^2(1-x)}.$$
Thus we have
$$\E^{1243}(L_n)=\frac{2^{n-3}(n^2-n+4)}{(n-1)2^{n-2}+1}$$
and
$$\E^{1243}(L_n^2)=\frac{2^{n-4}(n^3+5n+2)}{(n-1)2^{n-2}+1}.$$
\end{example}

\end{document}